\documentclass{article}
\usepackage{geometry}
\usepackage{latexsym,amsthm,amsmath,amssymb,url}
\usepackage[ruled, linesnumbered]{algorithm2e}
\usepackage[utf8]{inputenc}
\usepackage{tikz,authblk}
\usetikzlibrary{arrows.meta}
\usepackage{comment}
\usetikzlibrary{decorations.pathreplacing}
\usetikzlibrary{calc,shapes}

\newtheorem{theorem}{Theorem}[section]
\newtheorem{problem}{Problem}

\newtheorem{proposition}[theorem]{Proposition}

\newcommand{\nodeRR}{\tikzstyle{vertexR}=[regular polygon,regular polygon sides=4,draw, top color=gray!5, bottom color=gray!30, minimum size=12pt, scale=0.62, inner sep=0.8pt]
{\small \begin{tikzpicture}[scale=0.30]
\node (x1) at (0,0) [vertexR] {};
\end{tikzpicture} } }

\newcommand{\nodeYY}{\tikzstyle{vertexY}=[circle,draw, top color=gray!5, bottom color=gray!30, minimum size=11pt, scale=0.62, inner sep=0.8pt]
{\small \begin{tikzpicture}[scale=0.30]
\node (x1) at (0,0) [vertexY] {};
\end{tikzpicture} } }

\newcommand{\nodeWW}{\tikzstyle{vertexW}=[star, star points=10,draw, top color=gray!5, bottom color=gray!30, minimum size=11pt, scale=0.62, inner sep=0.8pt]
{\small \begin{tikzpicture}[scale=0.30]
\node (x1) at (0,0) [vertexW] {};
\end{tikzpicture} } }

\newcommand{\GG}[1]{{#1}}

\newcommand{\QG}[1]{{#1}}
\newcommand{\AYn}[1]{{#1}}
\newcommand{\GGn}[1]{{#1}}
\newcommand{\YZn}[1]{{#1}}
\newcommand{\QGn}[1]{{#1}}
\newcommand{\YLn}[1]{{#1}}
\newcommand{\GGnn}[1]{{#1}}

\newcommand{\wdom}[1]{\rightsquigarrow}
\newcommand{\wdomV}[1]{\rightsquigarrow_{#1}}

\title{Forward Arc Maximization for Hamilton Oriented Cycles and Paths in Generalizations of Tournaments}

\author{Qiwen Guo\thanks{Center for Combinatorics and LPMC, Nankai University. {\tt gqwmath@163.com}.}
\hspace{2mm}
Gregory Gutin\thanks{Department of Computer Science. Royal Holloway University of London, {\tt g.gutin@rhul.ac.uk}, and School of Mathematical Sciences and LPMC, Nankai University.}
\hspace{2mm} Yongxin Lan\thanks{School of Science, Hebei University of Technology, {\tt  2019110@hebut.edu.cn}.}
\\ Qi Shao\thanks{Center for Combinatorics and LPMC, Nankai University. {\tt sq1449682195@163.com}.}
\hspace{2mm} Anders Yeo\thanks{Department of Mathematics and Computer Science, University of Southern Denmark. {\tt andersyeo@gmail.com}, and Department of Mathematics, University of Johannesburg.}
\hspace{2mm} Yacong Zhou\thanks{Shenzhen Institutes of Advanced Technology. {\tt yacong.zhou96@gmail.com}.}}
\begin{document}
\maketitle
\begin{abstract}
Gishboliner, Krivelevich, and Michaeli (2023) conjectured the following generalization of Dirac's theorem: If the minimum degree $\delta$ of an $n$-vertex oriented graph $G$ is greater or equal to $n/2$, then $G$ has a Hamilton oriented cycle with at least $\delta$ forward arcs. Freschi and Lo (2024) proved this conjecture. In this paper, we study the problem of maximizing the number of forward arcs in Hamilton oriented cycles/paths in generalizations of tournaments. We obtain characterizations for the maximum number of forward arcs in semicomplete multipartite digraphs and locally semicomplete digraphs. These characterizations lead to polynomial-time algorithms. Note that the above problems are NP-hard for some other generalizations of tournaments \GGnn{even though the Hamilton cycle problem is polynomial-time solvable for these digraph classes}. 
\end{abstract}

\section{Introduction}

Terminology and notation not introduced in this paper can be found in \cite{BG00, BG18, BM08}. The {\em underlying graph} of a digraph $D$ is an undirected multigraph $U(D)$ obtained from $D$ by removing the orientations of all arcs in $D$. The {\em degree} of a vertex $x$ in a digraph $D$, denoted by $d_D(x)$, is the degree $d_{U(D)}(x)$ of $x$ in $U(D)$. The \emph{minimum degree} of a directed or undirected graph $H$ is denoted by $\delta(H)$. A digraph $D$ is an {\em oriented graph} if there is no more than one arc between any pair of vertices in $D$.  

\GGnn{An {\em oriented path} $P=v_1v_2\dots v_p$ of $D$ is a subdigraph $P$ of $D$ such that $U(P)$ is a path in $U(D)$. Similarly, one can define an {\em oriented cycle} in $D$. Note that the vertices of oriented paths and cycles are ordered (like in undirected graphs).} For an oriented cycle $C$, the arc between $v_i$ and $v_{i+1} \pmod{p}$ on $C$ is \emph{forward} if $v_iv_{i+1}\in A(C)$ and  \emph{backward} if $v_{i+1}v_{i}\in A(C)$. \GGnn{Similarly, one can define forward and backward arcs in an oriented path. In this paper, since we are looking for oriented paths and cycles with maximum number of forward arcs, we will always assume that in an oriented path $P=v_1v_2\dots v_p$ of $D$, the arc between $v_i$ and $v_{i+1}$ is forward as long as $v_iv_{i+1}\in A(D)$, and similarly for oriented cycles. The number of forward (backward, respectively) arcs in an oriented path or cycle $Q$ is denoted by $\sigma^+(Q)$ ($\sigma^-(Q)$, respectively). If all arcs of an oriented cycle $Q$ are forward, then $Q$ is a {\em directed cycle} (or, just a {\em cycle}).
Similarly, for directed paths. }

An oriented cycle (oriented path, respectively) in $D$ is \emph{Hamilton} if it contains all vertices in $D$. Note that a digraph $D$ has a Hamilton oriented cycle (path, respectively) if and only if $U(D)$ has a Hamilton cycle (path, respectively). We say that a digraph $D$ is \emph{hamiltonian} if it has a Hamilton cycle. Similar definitions and notation can be introduced for oriented paths.



\GGnn{In 1963, Erd{\H o}s \cite{Erd63} introduced the notion of graph discrepancy for 2-edge-colored undirected graphs. His paper has launched an extensive research on the topic, see e.g. references in \cite{FL24,GKM23}. Sixty years later, Gishboliner, Krivelevich, and Michaeli \cite{GKM23} introduced a digraph analog of graph discrepancy. The directed discrepancy introduced by Gishboliner et al.  is the maximum of $\sigma_{\max}(C)$ over all Hamilton oriented cycles $C$ of a digraph $D$ which has a Hamilton oriented cycle, where $\sigma_{\max}(C)=\max\{\sigma^+(C),\sigma^-(C)\}$. Note that, in fact, the maximum of $\sigma_{\max}(C)$ equals the maximum of $\sigma^+(C)$\footnote{Indeed, for an oriented cycle $C=v_1v_2\dots v_pv_1$, $\sigma_{\max}(C)=\max\{\sigma^+(C),\sigma^+(C^{-1})\},$ where $C^{-1}=v_1v_p\dots ,v_2v_1$.}, which makes the use of notation $\sigma_{\max}(C)$ unnecessary.}
Gishboliner et al.  conjectured the following generalization of Dirac's theorem \cite{D52}, which was recently confirmed by Freschi and Lo \cite{FL24}. 

\begin{theorem}\label{Dirac-type theorem}
	Let $D$ be an oriented graph on $n\ge 3$ vertices. If $\delta(D) \ge \frac{n}{2}$, then there exists a Hamilton oriented cycle $C$ in $D$ such that \GGnn{$\sigma^+(C)\ge \delta(D)$.}
\end{theorem}

\GGnn{It is natural, along with bounds on  directed discrepancy, to study the following optimization problem.}

\begin{problem}{\sc Maximum-Forward-Arc Hamilton Oriented Cycle/Path (MFAHOC/MFAHOP)}
	Given a digraph $D$, \GGnn{decide whether $D$ contains a Hamilton oriented cycle/path, and if it does, then}
	find a Hamilton oriented cycle/path with the greatest number of forward arcs.
\end{problem}

\GGnn{It is not hard to show that MFAHOC (and similarly, MFAHOP) is {\sf NP}-hard for $n$-vertex oriented graphs $D$ with $\delta(D) \ge \frac{n}{2}$. 
Indeed, consider an oriented graph $G$ with $n$ vertices and a vertex-disjoint tournament $T$ also with $n$ vertices. To obtain an oriented graph $D$, add all arcs from $V(G)$ to $V(T)$. Since $T$ has a Hamilton path, $D$ has a Hamilton oriented cycle with $2n-1$ forward arcs if and only if $G$ has a Hamilton path.}

\GGnn{In the rest of the paper, we will study MFAHOC/MFAHOP for generalizations of tournaments.}
\GG{A {\em semicomplete digraph} is a digraph obtained from a complete graph by replacing every edge $\{x,y\}$ with either arc $xy$ or arc $yx$ or two arcs $xy$ and $yx.$
It is well-known that every semicomplete digraph has a Hamilton path and every strongly connected semicomplete digraph has a Hamilton cycle, see e.g. \cite{BG18}. These two results make it straightforward to find a Hamilton oriented path (cycle, respectively) with maximum number of forward arcs in a semicomplete digraph $D.$

In Section \ref{sec:3}, we consider semicomplete multipartite digraphs, which are generalizations of semicomplete digraphs.} A {\em semicomplete $p$-partite (or, multipartite) digraph} is a digraph obtained from a complete $p$-partite graph ($p\ge 2$) by replacing every edge $\{x,y\}$ with either arc $xy$ or arc $yx$ or two arcs $xy$ and $yx.$ Maximal independent vertex sets of a semicomplete multipartite digraph are called its {\em partite sets}.  For a survey on paths and cycles in semicomplete multipartite digraphs, see \cite{Y18}.
In Section \ref{sec:3}, we prove two theorems which determine  the maximum of \GGnn{$\sigma^+(Q)$} over all Hamilton oriented paths (cycles, respectively) $Q$ in a semicomplete multipartite digraph $D$ provided $D$ has a Hamilton oriented path (cycle, respectively).
	
Let $D$ be a semicomplete multipartite digraph with partite sets of sizes $n_1,\dots ,n_p$. We say that $D$ satisfies the {\em HC-majority inequality} ( {\em HP-majority inequality}, respectively) if $2\max\{n_i:\ i\in [p]\} \le \sum_{i=1}^pn_i$ ($2\max\{n_i:\ i\in [p]\}\le (\sum_{i=1}^p n_i)+1,$ respectively). It is not hard to prove that a semicomplete multipartite digraph $D$ \GGnn{with at least 3 vertices and} with partite sets of sizes $n_1,\dots ,n_p$ has a Hamilton oriented cycle (path, respectively) if and only if $D$ satisfies the  HC-majority inequality (the  HP-majority inequality, respectively).

A {\em 1-path-cycle factor} in a digraph $H$ is  a spanning subdigraph of $H$ consisting of a \GGnn{non-empty} path and a \GGnn{(possibly, empty)} collection of cycles, all vertex-disjoint. Note that a Hamilton path
is a 1-path-cycle factor. \QGn{A {\em cycle factor} in a digraph $H$ is a spanning subdigraph of $H$ consisting of vertex-disjoint cycles.} We will use the following characterization of Gutin \cite{G88,G93} of semicomplete multipartite digraphs with Hamilton paths.

\begin{theorem}\label{mthp}
A semicomplete multipartite digraph  has a Hamilton path if and only if it contains a 1-path-cycle factor.  In polynomial time, one can decide whether a semicomplete multipartite digraph $D$ has a Hamilton path and find a Hamilton path in $D$, if it exists.
\end{theorem}
	
\GGnn{For a digraph $D$ we introduce the following very useful  notion.} The {\em symmetric (0,1)-digraph} of $D$ is a digraph $\hat{D}$ obtained from $D$ by assigning cost 1 to the arcs of $D$ and adding arc $yx$ of cost 0 for every $xy\in A(D)$
such that $yx\not\in A(D)$. The {\em cost}  of a subgraph $H$ of $\hat{D}$ is the sum of the costs of arcs of $H.$

The following result determines the maximum of \GGnn{$\sigma^+(P)$} over all Hamilton oriented path $P$ in a semicomplete multipartite digraph $D$ provided $D$ satisfies the HP-majority inequality.
Theorem \ref{mthop} generalizes Theorem \ref{mthp}, and is proved in Subsection~\ref{HOP}.

\begin{theorem}\label{mthop}
Let $D$ be an $n$-vertex semicomplete multipartite digraph satisfying the HP-majority inequality and let $c^{pc}_{\max}$ be the maximum cost of a 1-path-cycle factor in $\hat{D}.$
Let $\sigma^{hp}_{\max}$ be \GGnn{the maximum of $\sigma^+(P)$ over all Hamilton oriented paths $P$ in $D$.}
Then $\sigma^{hp}_{\max}=c^{pc}_{\max}$. Both $\sigma^{hp}_{\max}$ and a Hamilton oriented path $P$ with  $\sigma^+(P)=\sigma^{hp}_{\max}$
can be found in polynomial time.
\end{theorem}

While Bang-Jensen, Gutin and Yeo \cite{BGY98} proved that in polynomial time one can decide whether a semicomplete multipartite digraph has a Hamilton cycle, no characterization of hamiltonian semicomplete multipartite digraphs has been obtained. The next theorem determines $\sigma^+(C)$ for a Hamilton oriented cycle $C$ in a semicomplete multipartite digraph $D$ satisfying the HC-majority inequality.

\begin{theorem}\label{mthoc}
Let $D$ be an $n$-vertex semicomplete multipartite digraph \GGnn{with $n\ge 3$} satisfying the HC-majority inequality. \QG{Let $c^{cf}_{\max}$ be the maximum cost of a cycle factor in $\hat{D}.$} Let $\sigma^{hc}_{\max}$ be \GGnn{the maximum of $\sigma^+(C)$ over all Hamilton oriented cycles $C$ in $D$.}
Then $\sigma^{hc}_{\max}=c^{cf}_{\max}$ unless $c^{cf}_{\max}=n$ and $D$ is not hamiltonian, in which case $\sigma^{hc}_{\max}=n-1.$
Both $\sigma^{hc}_{\max}$ and a Hamilton oriented cycle $C$ with  $\sigma^+(C)=\sigma^{hc}_{\max}$ can be found in polynomial time.
\end{theorem}

Our proof of Theorem \ref{mthoc} uses the following theorem, which may be of independent interest.

\begin{theorem} \label{ham_path_diff_ends}
Let $D$ be a semicomplete multipartite digraph and let ${F}$ be a 1-path-cycle factor in $D$.
Let $P$ denote the path in ${F}$ and assume the initial and terminal vertices in $P$ belong to different partite sets.
Then $D$ contains a Hamilton path whose initial and terminal vertices belong to different partite sets.
\end{theorem}

In turn, our proof of Theorem \ref{ham_path_diff_ends} is based on a structural result of Yeo \cite{Y97}, \GGnn{see Theorem \ref{main_irreducible_simple} in the next section.}

\vspace{1mm}

\GGn{For a digraph $D$ and $x\in V(D)$, let $N^+(x)=\{y\in V(D):\ xy\in A(D)\}$ and $N^-(x)=\{y\in V(D):\ yx\in A(D)\}.$
In Section \ref{sec:lsd} we consider locally semicomplete digraphs.
A digraph $D$ is {\em locally semicomplete} if for every $x\in V(D)$, both $D[N^+(x)]$ and $D[N^-(x)]$ are semicomplete digraphs.
A digraph $D$ is connected if $U(D)$ is connected.
A digraph $D$ is {\em strong} if there is a path from $x$ to $y$ for every ordered pair $x,y$ of vertices of $D$.
A {\em strong component} of a digraph $D$ is a maximal strong subgraph of $D$.
It is well known \cite{B90,B} that every connected locally semicomplete digraph has a Hamilton path.
This implies that the problem of finding a Hamilton oriented path is trivial.
Also, it is well known \cite{B90,B} that every strong semicomplete digraph has a Hamilton cycle.

However, this does not imply a characterization of Hamilton oriented cycles with \AYn{the maximum possible} number of forward arcs.
We provide such a characterization in Section \ref{sec:lsd}.
It is easy to see that we can order strong components $H_1,\dots ,H_{\ell}$ of a non-strongly connected digraph $H$  such that there is no arc from \AYn{$H_j$ to $H_i$} for $i<j$.
This is called an {\em acyclic ordering} of strong components of $H$.
For a non-strongly connected locally semicomplete digraph, it is well known \cite{B90,B} that an acyclic ordering of strong components is unique and has a few nice properties described in Section \ref{sec:lsd}.

Let $S,T$ be vertex-disjoint subgraphs of a digraph $H$. A path $P=p_1\dots p_t$ of $H$ is an $(S,T)$-{\em path} if $p_1\in \YLn{V(S)}$, $p_t\in \YLn{V(T)}$ and all other vertices of $P$ are outside of $S\cup T$. Let $d(S,T)$ denote the length of a shortest $(S,T)$-path.
Here is our characterization.


\begin{theorem}\label{thm: hoc in LSD}
Let $D$ be a connected locally semicomplete digraph on $n\ge 3$ vertices. Let $C_1, C_2,\dots, C_{\ell}$ be the acyclic ordering of strong components of $D$.
\GGnn{Let $\sigma^{hc}_{\max}=0$ if $D$ has no Hamilton oriented cycle}, and let $\sigma^{hc}_{\max}$ be the \GGnn{maximum of $\sigma^+(R)$ over all Hamilton oriented cycles} $R$ in $D$, \GGnn{otherwise}. The following now holds.

\begin{itemize}
\item If $D$ is not strong and $U(D)$ is 2-connected, then $\sigma^{hc}_{\max}= n - d(C_1,C_{\ell})$.
\item If $D$ is not strong and $U(D)$ is not 2-connected, then \GGnn{$\sigma^{hc}_{\max}=0.$}
\item If $D$ is strong, then $\sigma^{hc}_{\max}= n$.
\end{itemize}

Furthermore, in polynomial time, we can find a Hamilton oriented cycle of $D$ with the maximum number of forward arcs \GGnn{provided $D$ has such an oriented cycle}. 
\end{theorem}
}

\GGn{We conclude the paper in Section \ref{sec:disc}, where we discuss the complexity of finding a Hamilton oriented cycle with maximum number of forward arcs
in other classes of digraphs. }

\section{\GGn{Hamilton oriented paths and cycles with maximum number of forward arcs in semicomplete multipartite digraphs}}\label{sec:3}
The following result gives a characterization of semicomplete multipartite digraphs with a Hamilton oriented cycle (path, respectively).
\begin{proposition}
A semicomplete multipartite digraph $D$ \GGnn{with at least 3 vertices and} with partite sets of sizes $n_1,n_2,\dots ,n_p$ has a Hamilton oriented cycle (path, respectively) if and only if $D$ satisfies the  HC-majority inequality (the  HP-majority inequality, respectively).
\end{proposition}
\begin{proof}
Observe that $D$ has a Hamilton oriented cycle if and only if $K_{n_1,n_2,\dots ,n_p}$ has a Hamilton cycle. Note that if $K_{n_1,n_2,\dots ,n_p}$ does not satisfy the  HC-majority inequality, then it has no Hamilton cycle. If $K_{n_1,n_2,\dots ,n_p}$ satisfies the  HC-majority inequality, then it has a Hamilton cycle by Dirac's theorem \cite{D52}.

Observe that $D$ has a Hamilton oriented path if and only if the digraph $D'$ obtained from $D$ by adding a new vertex $x$ and all arcs of the form $xv,vx$ for $v\in V(D)$ has a Hamilton oriented cycle which is if and only if $K_{n_1,n_2,\dots ,n_p,1}$ has a Hamilton cycle. Thus, $D$ has a Hamilton oriented path if and only if it satisfies the HP inequality.
\end{proof}


\subsection{Hamilton Oriented Paths} \label{HOP}

Let $D$ be a digraph. Recall that its {\em symmetric (0,1)-digraph} is a digraph $\hat{D}$ obtained from $D$ by assigning cost 1 to the arcs of $D$ and adding arc $yx$ of cost 0 for every $xy\in A(D)$
such that $yx\not\in A(D)$.
The {\em cost}  of a subgraph $H$ of $\hat{D}$ is the sum of the costs of arcs of $H.$ The following theorem characterizes the discrepancy of Hamilton oriented paths in semicomplete multipartite digraphs by the cost of path-cycle factors in its symmetric (0,1)-digraph.\\

	\noindent{\bf Theorem \ref{mthop}.}
    \textit{Let $D$ be a semicomplete multipartite digraph on $n$ vertices, satisfying the HP-majority inequality and let $c^{pc}_{\max}$ be the maximum cost of a 1-path-cycle factor in $\hat{D}.$
Let $\sigma^{hp}_{\max}$ be \GGnn{the maximum of $\sigma^+(P)$ over all Hamilton oriented paths} $P$ in $D$.
Then $\sigma^{hp}_{\max}=c^{pc}_{\max}$. Both $\sigma^{hp}_{\max}$ and a Hamilton oriented path $P$ with  \GGnn{$\sigma^+(P)=\sigma^{hp}_{\max}$}
can be found in polynomial time. }
\begin{proof}
Note that a Hamilton path $P=x_1x_2\dots x_n$ in $\hat{D}$ can be transformed into a Hamilton oriented path of $D$ by replacing every arc $x_ix_{i+1}$ of zero-cost by $x_{i+1}x_i$.
Also by the construction of $\hat{D}$, $\sigma^{hp}_{\max}$ is equal to the maximum cost of a Hamilton path in $\hat{D}.$

Observe that a Hamilton path of $\hat{D}$ is a 1-path-cycle factor of $\hat{D}$. Thus, the maximum cost of a Hamilton path of $\hat{D}$ is smaller or equal to $c^{pc}_{\max}.$
 Let $F$ be a 1-path-cycle factor of $\hat{D}$ of maximum cost and let $D_F$ be a spanning subdigraph of $\hat{D}$ with $A(D_F)=A(D)\cup A(F).$
 By Theorem \ref{mthp}, $D_F$ has a Hamilton path $P$. Note that the number of zero-cost arcs in $P$ cannot be larger than that in $F$, which means that the cost of $P$
 is not smaller than that of $F$. Thus,  the maximum cost of a Hamilton path of $\hat{D}$ is equal to $c^{pc}_{\max}$.  

By Theorem \ref{mthp} and the proof above, given a maximum cost 1-path-cycle factor in $\hat{D},$ we can find both $\sigma^{hp}_{\max}$ and a Hamilton oriented path $P$ with  \GGnn{$\sigma^+(P)=\sigma^{hp}_{\max}$} in polynomial time. To complete the proof of this theorem, we will describe how one can find  a maximum cost 1-path-cycle factor in $\hat{D}$ in polynomial time. Let $\hat{D}'$ be obtained from $\hat{D}$ by exchanging costs: 0 to 1 and 1 to 0 simultaneously. Observe that a maximum cost 1-path-cycle factor in $\hat{D}$ is a minimum cost 1-path-cycle factor in $\hat{D}'$ and vice versa. \GGn{By a remark in the last paragraph of Sec. 4.11.3 of \cite{BG00} using minimum cost flows, one can find a minimum cost 1-path-cycle factor in an arc-weighted digraph in polynomial time. Since details on how to do it are omitted in \cite{BG00}, we give them below.}

Add new vertices $s$ and $t$ to $\hat{D}'$ together with arcs $sv$ and $vt$ of cost 0 for all vertices $v\in V(D)$ and assign lower and upper bound 1 to every vertex in $V(D)\cup \{s,t\}.$ In this network $N$ in polynomial time, we can find a minimum cost $(s,t)$-flow $f_{\min}$ of value 1.
Note that all arcs of $N-\{s,t\}$ with flow of value 1 form a minimum cost 1-path-cycle factor in $\hat{D}'$ which is a maximum cost 1-path-cycle factor in $\hat{D}.$
\end{proof}


\subsection{Hamilton Oriented Cycles}

Recall the following theorem which is important for the main result of this section.

\begin{theorem}[\cite{BGY98,Y99}]\label{mthc}
There is a polynomial-time algorithm for deciding whether a semicomplete multipartite digraph $D$ has a Hamilton cycle and if $D$ has one, then find it.
\end{theorem}

If a vertex $v$ belongs to a cycle $C$ then we denote the successor of $v$ on the cycle by $v_C^+$ and the predecessor by $v_C^-$. When $C$ is clear from the context,
we may omit the subscript $C$.
Let $C_1$ and $C_2$ be two disjoint cycles in a semicomplete multipartite digraph $D$. \GGnn{Let $V_1,\dots ,V_p$ be partite sets of $D$.} Let  the following hold for some partite set $V_i$.

\vspace{1mm}

{\em For every arc $u_2 v_1$ from $C_2$ to $C_1$ we have $\{u_2^+,v_1^-\} \subseteq V_i$, where $u_2^+$ is the successor of $u_2$ on $C_2$ and
$v_1^-$ is the predecessor of $v_1$ on $C_1$. }

\vspace{1mm}

\noindent Then we say that $C_1$ {\em $V_i$-weakly-dominates} $C_2$ and denote this by $C_1 \wdomV{V_i} C_2$.
If $C_1 \wdomV{V_i} C_2$ for some $i$ then we also say that
$C_1$ {\em weakly-dominates} $C_2$ and denote this simply by $C_1 \wdom{} C_2$.

See Figure~\ref{figIrreducible1} for an illustration of this definition. For example, in Figure~\ref{figIrreducible1}, $w_1y_2$ is the only arc
from $C_3$ to $C_2$ and $\{w_1^+,y_2^-\} \in V_3$ (as $w_1^+ = w_2$ and $y_2^- = y_1$). Therefore, $C_2 \wdomV{V_3} C_3$.

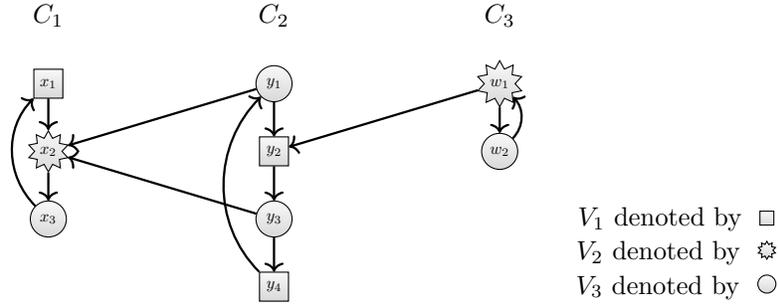
\begin{figure}[h]
\begin{center}
\tikzstyle{vertexY}=[circle,draw, top color=gray!5, bottom color=gray!30, minimum size=22pt, scale=0.62, inner sep=0.8pt]
\tikzstyle{vertexZ}=[star, star points=7,star point ratio=0.8,draw, top color=gray!5, bottom color=gray!30, minimum size=25pt, scale=0.62, inner sep=0.8pt]
\tikzstyle{vertexW}=[star, star points=10,draw, top color=gray!5, bottom color=gray!30, minimum size=25pt, scale=0.62, inner sep=0.8pt]
\tikzstyle{vertexX}=[diamond,draw, top color=gray!5, bottom color=gray!30, minimum size=25pt, scale=0.62, inner sep=0.8pt]
\tikzstyle{vertexQ}=[regular polygon,regular polygon sides=5,draw, top color=gray!5, bottom color=gray!30, minimum size=25pt, scale=0.62, inner sep=0.8pt]
\tikzstyle{vertexR}=[regular polygon,regular polygon sides=4,draw, top color=gray!5, bottom color=gray!30, minimum size=25pt, scale=0.62, inner sep=0.8pt]
\tikzstyle{vertexS}=[regular polygon,regular polygon sides=3,draw, top color=gray!5, bottom color=gray!30, minimum size=25pt, scale=0.62, inner sep=0.8pt]
\begin{tikzpicture}[scale=0.30]

\node at (2,13) {$C_1$};
\node at (12,13) {$C_2$};
\node at (22,13) {$C_3$};

\node (x11) at (2.0,10.0) [vertexR] {$x_1$};   
\node (x12) at (2.0,7.0) [vertexW] {$x_2$};    
\node (x13) at (2.0,4.0) [vertexY] {$x_3$};    

\node (x21) at (12.0,10.0) [vertexY] {$y_1$};
\node (x22) at (12.0,7.0) [vertexR] {$y_2$};
\node (x23) at (12.0,4.0) [vertexY] {$y_3$};
\node (x24) at (12.0,1.0) [vertexR] {$y_4$};

\node (x31) at (22.0,10.0) [vertexW] {$w_1$};
\node (x32) at (22.0,7.0) [vertexY] {$w_2$};

\draw [->, line width=0.03cm] (x11) -- (x12);
\draw [->, line width=0.03cm] (x12) -- (x13);
\draw [->, line width=0.03cm] (x13) to [out=135, in=225] (x11);

\draw [->, line width=0.03cm] (x21) -- (x22);
\draw [->, line width=0.03cm] (x22) -- (x23);
\draw [->, line width=0.03cm] (x23) -- (x24);
\draw [->, line width=0.03cm] (x24) to [out=135, in=225] (x21);

\draw [->, line width=0.03cm] (x31) -- (x32);
\draw [->, line width=0.03cm] (x32) to [out=45, in=315] (x31);

\draw [->, line width=0.03cm] (x21) -- (x12);
\draw [->, line width=0.03cm] (x23) -- (x12);
\draw [->, line width=0.03cm] (x31) -- (x22);

\node at (30,4) {$V_1$ denoted by \nodeRR{}};
\node at (30,2.5) {$V_2$ denoted by \nodeWW{}};
\node at (30,1) {$V_3$ denoted by \nodeYY{}};
\end{tikzpicture}
\end{center}
\caption[An illustration]{Arcs from cycle $C_i$ to $C_j$, for $1 \leq i < j \leq 3$ are not shown.
Note that $C_1 \wdomV{V_1} C_2$
and $C_2 \wdomV{V_3} C_3$.
As there are no arcs from $C_3$ to $C_1$, we have $C_1 \wdomV{V_i} C_3$ for all $i=1,2,3$.}
\label{figIrreducible1}
\end{figure}


We will use the following result which is a corollary of Theorem 3.13 in  \cite{Y97} by Yeo.

\begin{theorem}\label{main_irreducible_simple}
Let $D$ be a semicomplete multipartite digraph. \GGnn{In polynomial time, we can decide whether $D$ has a cycle factor and if it has one, then, in polynomial time, we can find either a Hamilton cycle  or a cycle factor with at least $t\ge 2$ cycles in which we can order its cycles $C_1,C_2,\ldots, C_t$ such that 
$C_i \wdom{} C_j$ for all $1 \leq i < j \leq t$.}
\end{theorem}

For a cycle $C=x_1x_2\dots x_px_1$ and $i,j\in [p]$, we write $C[x_i,x_j]$ to denote the path $x_ix_{i+1}\dots x_j \pmod{p}$. By the existence of the \QG{ordered} cycle factor described in Theorem \ref{main_irreducible_simple}, we may obtain the existence of a Hamilton path with special end-vertices in a semicomplete multipartite digraph.\\

\noindent\textbf{Theorem \ref{ham_path_diff_ends}.}
\textit{Let $D$ be a semicomplete multipartite digraph and let ${F}$ be a 1-path-cycle factor in $D$.
Let $P$ denote the path in ${F}$ and assume the initial and terminal vertices in $P$ belong to different partite sets.
Then $D$ contains a Hamilton path whose initial and terminal vertices belong to different partite sets.}

\begin{proof}

Let $D$, ${F}$ and $P$ be defined as in the statement of the theorem. Let $P=p_1 p_2 p_3 \ldots p_{\ell}$ and
let $D'$ be obtained from $D$ by adding the arc $p_{\ell} p_1$ if it does not exist in $D$ already (if $p_{\ell} p_1 \in A(D)$ then $D'=D$).
Note that $D'$ has a cycle factor. 

By Theorem~\ref{main_irreducible_simple}, we can find in polynomial time a cycle factor $C$  of $D'$ with $t$ cycles and either $t=1$ or we can order its cycles, $C_1,C_2,\ldots,C_t$, of ${C}$ such that
$C_i \wdom{} C_j$ for all $1 \leq i < j \leq t$, where $t \geq 2$. If the arc $p_{\ell} p_1$ belongs to a cycle in ${C}$ then let $a=p_{\ell} p_1$ and otherwise let $a$ be an arbitrary arc on a cycle in ${C}$.
Assume that the arc $a$ belongs to $C_r$, where $1 \leq r \leq t$ and consider the $1$-path-cycle factor, ${F}'$, in $D$ obtained from
${C}$ by deleting the arc $a$. Let $P'=p_1' p_2' p_3' \ldots p_m'$ be the path in ${F}'$.

Note that $p_1'$ and $p_m'$ belong to different partite sets, and assume without loss of generality that $p_1' \in V_1$ and $p_m' \in V_2$, where
$V_1,V_2,\ldots,V_c$ are the partite sets of $D$.

First consider the case when $r<t$. We now transform $P'$ and $C_{r+1}$ into a path where the end-points belong to different partite sets as follows.
We first consider the case when there is no arc from $C_{r+1}$ to $p_m'$ in $D$. In this case let $y \in V(C_{r+1}) \cap V_1$ if such a vertex exists and otherwise let $y \in V(C_{r+1}) \setminus V_2$ be arbitrary. Note that $p_m' y \in A(D)$ and the path $P' C_{r+1}[y,y^-]$ is a path with vertex set $V(C_r) \cup V(C_{r+1})$ and with end-points in different partite sets.

Secondly we consider the case when there is an arc from $C_{r+1}$ to $p_m'$ in $D$, say $z p_m'$.
By Theorem~\ref{main_irreducible_simple} we note that $p_{m-1}'$ and $z^+$ both belong to the same partite set $V_i$ \GGnn{and $p'_{m-1}z,p'_mz^+\in A(D).$} 
Hence, the following are both paths on the vertex set $V(C_r) \cup V(C_{r+1})$:
\[
P_1 = p_1' p_2' \cdots p_{m-1}' z p'_m C_{r+1}[z^+,z^-] \mbox{ and }
P_2 = p_1' p_2' \cdots p_{m}' C_{r+1}[z^+,z]
\]
And as $z$ and $z^-$ belong to different partite sets either $P_1$ or $P_2$ has initial and terminal vertices from different partite sets. So we have in all cases
found a path with vertex set $V(C_r) \cup V(C_{r+1})$ with  initial and terminal vertices in different partite sets.

We can repeat this process in order to get a path with vertex set $V(C_r) \cup V(C_{r+1}) \cup V(C_{r+2})$ with  initial and terminal vertices in different partite sets.
And continuing this process further we obtain a path with vertex set $V(C_r) \cup V(C_{r+1}) \cup \cdots \cup V(C_t)$ with  initial and terminal vertices in different partite sets.

We can now analogously merge cycles $C_{r-1}$, $C_{r-2}$, $\ldots$, $C_1$ with the path ending up with a Hamilton path in $D$ where the initial and terminal vertices are in different partite sets.
\end{proof}
Now we characterize the oriented discrepancy of semicomplete multipartite digraphs by the cost of cycle factors in its symmetric (0,1)-digraph.\\

\noindent\textbf{Theorem \ref{mthoc}.}
\textit{Let $D$ be an $n$-vertex semicomplete multipartite digraph  \GGnn{with $n\ge 3$} satisfying the HC-majority inequality and \QG{let $c^{cf}_{\max}$ be the maximum cost of a cycle factor in $\hat{D}.$ }Let $\sigma^{hc}_{\max}$ be the \GGnn{maximum $\sigma^+(C)$ over all Hamilton oriented cycles $C$} in $D$.
Then $\sigma^{hc}_{\max}=c^{cf}_{\max}$ unless $c^{cf}_{\max}=n$ and $D$ is not hamiltonian, in which case $\sigma^{hc}_{\max}=n-1.$
Both $\sigma^{hc}_{\max}$ and a Hamilton oriented cycle $C$ with  \GGnn{$\sigma^+(C)=\sigma^{hc}_{\max}$} can be found in polynomial time.}

\begin{proof}
\QG{Note that a Hamilton cycle $C=x_1x_2\dots x_nx_1$ in $\hat{D}$ can be transformed into a Hamilton oriented cycle of $D$ by replacing every arc $x_ix_{i+1} \pmod{n} $ of zero-cost by $x_{i+1}x_i \pmod{n}$.
Also by \GGn{the construction} of $\hat{D}$, $\sigma^{hc}_{\max}$ is equal to the maximum cost of a Hamilton cycle in $\hat{D}.$

Observe that a Hamilton cycle of $\hat{D}$ is a cycle factor of $\hat{D}$. Thus, the maximum cost of a Hamilton cycle of $\hat{D}$ is no more than $c^{cf}_{\max}.$
 Let $F$ be a cycle factor of $\hat{D}$ of maximum cost, \GGn{which can be found in polynomial time, see the last two paragraphs of Theorem \ref{mthop}}.
 Let $D_F$ be a spanning subdigraph of $\hat{D}$ with $A(D_F)=A(D)\cup A(F).$} We now consider the following cases.

\vspace{2mm}
{\bf Case 1:} $c^{cf}_{\max}<n$.
\vspace{2mm}
Thus, $F$ has an arc, $a$, of cost 0. Note that $F-a$ is a 1-path-cycle factor in the semicomplete multipartite digraph $D_F-a$.
By Theorem~\ref{ham_path_diff_ends} we note that $D_F-a$ contains a Hamilton path, $P=p_1 p_2 p_3 \cdots p_n$, with initial and terminal vertices from different partite sets.
Now \QG{for the Hamilton oriented cycle $C'=p_1 p_2 p_3 \cdots p_np_1$ in $D$, we know that $\sigma^+(C')$ is at least the cost of $P$ and this is at least the cost of $F$} (as the number of zero-cost arcs in $D_F-a$ is one smaller than  the number of zero-cost arcs in $F$). This implies that  $\sigma^{hc}_{\max}=c^{cf}_{\max}$ and completes the proof in this case.

\vspace{2mm}
{\bf Case 2:} $c^{cf}_{\max}=n$.
\vspace{2mm}
Then $F$ is a cycle factor of $D$. If $D$ is \YLn{hamiltonian} then clearly $\sigma^{hc}_{\max}=n=c^{cf}_{\max}$.
So we may assume that $D$ is not \YLn{hamiltonian}. But removing an arc from $F$ and using Theorem~\ref{ham_path_diff_ends} on the resulting
$1$-path-cycle factor in $D$ gives us a Hamilton path in $D$ where the initial and terminal vertices are in different partite sets.
 We therefore obtain a Hamilton oriented cycle in $D$ with at most 1 backward arc. As $D$ is not \YLn{hamiltonian}, we therefore get
$\sigma^{hc}_{\max}=n-1$, which completes the proof in this case.

The proofs of this theorem and Theorem \ref{ham_path_diff_ends} can be converted to corresponding polynomial-time algorithms. These algorithms and the
 polynomial-time algorithms of Theorems \ref{mthc} and \ref{main_irreducible_simple} imply that both $\sigma^{hc}_{\max}$ and a Hamilton oriented cycle $C$ with  \GGnn{$\sigma^+(C)=\sigma^{hc}_{\max}$} can be found in polynomial time.
\end{proof}
	
\section{Hamilton oriented cycles with maximum number of forward arcs in locally semicomplete digraphs}\label{sec:lsd}

Let $D$ be a digraph. For disjoint subsets $A$ and $B$ of $V(D)$, if all arcs exist from $A$ to $B$, we say $A$ {\em dominates} $B$.  We will use the following three results by
Bang-Jensen \cite{B90}.

\begin{theorem}\cite{B90}\label{thm: non-strong decomposition of LSD}
    Let $D$ be a connected locally semicomplete digraph that is not strong. Then the strong components of $D$ can be ordered uniquely as $C_1, C_2, \dots, C_{\ell}$ such that there is no arc from $V(C_j)$ to $V(C_i)$ when $j>i$, $V(C_i)$ dominates $V(C_{i+1})$ for each $i\in [\ell -1]$ and $C_t$ is a semicomplete digraph for each $t \in [\ell ]$. And if there is an arc from $V(C_i)$ to $V(C_k)$, then $V(C_i)$ dominates $V(C_j)$ for all $j=i+1,i+2, \dots,k$ and $V(C_t)$ dominates $V(C_k)$ for all $t=i,i+1,\dots,k-1$. \end{theorem}

\begin{theorem}\cite{B90}\label{thm: hp in LSD}
    A connected locally semicomplete digraph has a Hamilton path.
\end{theorem}

\begin{theorem}\cite{B90}\label{thm: hc in LSD}
    A strong locally semicomplete digraph has a Hamilton cycle.
\end{theorem}

Note that an acyclic ordering of strong components in a digraph can \YLn{be obtained} in polynomial time \cite{BG00}. Thus, the ordering in Theorem \ref{thm: non-strong decomposition of LSD} can be obtained in polynomial time. Using this ordering, it is not hard to construct a Hamilton path in a connected local semicomplete digraph \cite{B90}. The proof of Theorem \ref{thm: hc in LSD} in \cite{B90} leads to a polynomial-time algorithm for finding a Hamilton cycle \AYn{in} a strong locally semicomplete digraph. Finally, for vertex-disjoint subgraphs $S$ and $T$ of a digraph $D$, it is easy to see that we can find  a shortest $(S,T)$-path in polynomial time.

\vspace{1mm}

\begin{theorem}\label{thm: upper bound of LSD}
Let $D$ be a connected locally semicomplete digraph on $n\ge 3$ vertices that is not strong. Let $C_1, C_2,\dots, C_{\ell}$ be the acyclic ordering of strong components of $D$. If $U(D)$ is 2-connected, then there is a shortest $(C_1,C_{\ell})$-path $P$ and a Hamilton oriented cycle $R$ in $D$, such that $P$ is a subpath of $R$ and the set of backward arcs in $R$ is $A(P)$. Such a Hamilton oriented cycle $R$ can be constructed in polynomial time.
\end{theorem}

\begin{proof}
Let $D$ be a locally semicomplete digraph on $n\ge 3$ vertices that is not strong, but such that $U(D)$ is 2-connected. Let $C_1, C_2,\dots, C_{\ell}$ be the acyclic ordering of strong components of $D$ and define the function $cn$ such that $cn(x) = r$ if $x \in V(C_r)$.
Define the $(C_1,C_{\ell})$-path $P=p_1 p_2 \cdots p_q$ as follows. Let $p_1 \in V(C_1)$ be arbitrary and for each $i$ let $p_{i+1}$ be an arbitrary vertex such that $p_i p_{i+1} \in A(D)$ and $cn(p_{i+1})$ is maximum possible. Continue this process until $p_q \in C_{\ell}$. By Theorem~\ref{thm: non-strong decomposition of LSD} we note that $P$ has length $d(C_1,C_{\ell})$, i.e., it is a shortest $(C_1,C_{\ell})$-path in $D$.

We will now show that $D'=D - \{p_2,p_3,\ldots, p_{q-1}\}$ is connected. For the sake of contradiction assume this is not the case and let $y$ be a vertex which can not be reached from $p_1$ in \GGnn{$U(D')$} and such that $cn(y)$ is minimum.
This implies that $V(C_{cn(y)-1} )= \{p_i\}$ for some $i \in \{2,3,\dots,q-1\}$, as otherwise there is a vertex in $V(D') \cap V(C_{cn(y)-1})$ which dominates $y$ in $D'$, a contradiction to our choice of $y$.
As \GGnn{$U(D)$} is $2$-connected $p_i$ is not a cut vertex in \GGnn{$U(D)$}, which by Theorem~\ref{thm: non-strong decomposition of LSD} implies that $V(C_{cn(y)-2})$ dominates $V(C_{cn(y)})$. By our choice of $y$ this implies that $V(C_{cn(y)-2})=\{p_{i-1}\}$.
However, this contradicts our construction of $P$ as $p_{i-1}$ has an arc to $y$ and $cn(y)>cn(p_i)$. Therefore $D'$ is connected.

As $D'$ is connected, Theorem~\ref{thm: hp in LSD} implies that $D'$ contains a Hamilton path $Q$.
As $Q$ first picks up all vertices in $C_1$ and $C_1$ is a strong semicomplete digraph and therefore contains a Hamilton cycle, or is a single vertex, we may assume that $Q$ starts in $p_1$ (as every vertex in $C_1$ have the same out-neighbours in $D-V(C_1)$). Analogously we may assume that $Q$ ends in $p_q$.
Now \YZn{$R=p_1Qp_qp_{q-1}\cdots p_1$} is the desired oriented Hamilton cycle where all arcs on $P$ are backward arcs and all arcs on $Q$ are forward arcs.

Note that using the complexity remarks given after Theorem \ref{thm: hc in LSD}, the proof of this theorem can be converted into a polynomial-time algorithm for constructing $R.$
    \end{proof}

\begin{theorem}\label{thm: lower bound of LSD}
  \YLn{Let $D$ be a connected locally semicomplete digraph that is not strong.} Let $C_1, C_2,\dots, C_{\ell}$ be the acyclic ordering of strong components of $D$.
  For any oriented $(C_1,C_{\ell})$-path $P$, there is at least $d(C_1,C_{\ell})$ forward arcs in $P$.
    \end{theorem}

   \begin{proof}
   Let $D$ and $C_1, C_2,\dots, C_{\ell}$ be defined as in the theorem and furthermore define the function $cn$ such that $cn(x) = r$ if $x \in V(C_r)$.
   Let $P=p_1 p_2 p_3 \cdots \YZn{p_t}$ be an oriented $(C_1,C_{\ell})$-path in $D$.
   \YZn{Let $j_1=2$.  The arc $p_{j_1-1}p_{j_1}=p_1p_2$ is a forward arc in $P$ since $P$ is a $(C_1,C_l)$-path and vertices in $C_1$ do not have in-neighbours in $V(D)\setminus V(C_1)$. For $k\geq 2$, let $j_k$ be the smallest subscript greater than $j_{k-1}$ such that $cn(p_{j_{k-1}})<cn(p_{j_k})$ (such $j_k$ exists if $j_{k-1}\neq t$ since then $cn(p_{j_{k-1}})< cn(p_t)$). Observe, from how we choose $j_k$, that $cn(p_{j_k-1})\leq cn(p_{j_{k-1}})<cn(p_{j_k})$. \AYn{Continue this process and assume} that we end up obtaining a sequence $j_1,j_2,\dots, j_q$ for some positive integer $q$ where $j_q=t$. Since $cn(p_{j_{k}-1}) \leq cn(p_{j_{k-1}}) < cn(p_{j_{k}})$, by Theorem \ref{thm: non-strong decomposition of LSD}, we have that for all $k=2,3,4,\dots,q$, \YLn{$p_{j_{k}-1}p_{j_k}\in A(P)$} which in turn gives us $p_{j_{k-1}}p_{j_{k}}\in A(D)$. Thus, $p_1 p_{j_1} p_{j_2}  \cdots p_{j_q}$ is a $(C_1,C_{\ell})$-path in $D$ with $q$ arcs and therefore $q \geq d(C_1,C_{\ell})$. In addition, as $p_{j_{k}-1}p_{j_k}\in A(P)$ for all $k=1,2,\dots,q$, there are at least $q$ forward arcs in $P$, which combining with $q \geq d(C_1,C_{\ell})$ proves the result.}
   \end{proof}

   \noindent\textbf{Theorem \ref{thm: hoc in LSD}.}
   \textit{\GGnn{Let $\sigma^{hc}_{\max}=0$ if $D$ has no Hamilton oriented cycle}, and let $\sigma^{hc}_{\max}$ be the \GGnn{maximum of $\sigma^+(R)$ over all Hamilton oriented cycles} $R$ in $D$, \GGnn{otherwise}. The following now holds.
\begin{itemize}
\item \textit{If $D$ is not strong and $U(D)$ is 2-connected, then $\sigma^{hc}_{\max}= n - d(C_1,C_{\ell})$. }
\item \textit{If $D$ is not strong and $U(D)$ is not 2-connected, then \GGnn{ $\sigma^{hc}_{\max}= 0$.}}
\item \textit{If $D$ is strong, then $\sigma^{hc}_{\max}= n$.}
\end{itemize}
}
\textit{Furthermore, in polynomial time, we can find a Hamilton oriented cycle of $D$ with the maximum number of forward arcs \GGnn{provided $D$ has such an oriented cycle}.
}

   \begin{proof}
  We first consider the case when $D$ is not strong and $U(D)$ is 2-connected.
  Theorem \ref{thm: upper bound of LSD} implies that $D$ contains an oriented Hamilton cycle $R'$, which consists of two internally disjoint paths \GGn{from a vertex in $C_1$ to a vertex in $C_{\ell}$} such that one of them is a shortest $(C_1,C_{\ell})$-path in $D$. Therefore, \GGnn{$\sigma^{hc}_{\max} \ge \sigma^+(R') = n - d(C_1,C_{\ell})$.}

        Now let $R=c_1 c_2 \cdots c_n c_1$ be an arbitrary oriented Hamilton cycle in $D$ with the maximum possible number of forward arcs.
Without loss of generality, we may assume that $R[c_a,c_b]$ is a $(C_{\ell},C_1)$-path in $U(D)$, for some $a,b \in [n]$.
 By Theorem \ref{thm: lower bound of LSD}, we note that
the reverse of $R[c_a,c_b]$ contains at least $d(C_1,C_{\ell})$ forward arcs and therefore $R[c_a,c_b]$ contains at least $d(C_1,C_{\ell})$ backward arcs.
This implies that $R$ contains at least $d(C_1,C_{\ell})$ backward arcs, which by our choice of $R$ implies that
 \GGnn{$\sigma^{hc}_{\max} = \sigma^+(R) \leq n - d(C_1,C_{\ell})$}. So, $\sigma^{hc}_{\max} = n - d(C_1,C_{\ell})$, which completes the first case.

When $D$ is not strong and $U(D)$ is not 2-connected it is clear that $D$ contains no oriented Hamilton cycle, as deleting a vertex from a Hamilton cycle leaves the remaining graph connected. So we may now consider the case when $D$ is strong. However as every strong locally semicomplete digraph contains a Hamilton cycle by Theorem~\ref{thm: hc in LSD} we note that this case also holds.
Finally, the polynomial-time algorithm claims after Theorem~\ref{thm: hc in LSD}  and in  Theorem \ref{thm: upper bound of LSD} implies the polynomial-time algorithm claim of this theorem.
   \end{proof}

\section{Complexity of finding Hamilton oriented cycles with maximum number of forward arcs}\label{sec:disc}

\GGn{Encouraged by our results on semicomplete multipartite digraphs and locally semicomplete digraphs, one may guess that if the Hamilton cycle problem is polynomial-time solvable in a class of generalizations of semicomplete digraphs, then
the same is true for the problem of finding a Hamilton \QG{oriented} cycle with maximum number of forward arcs.} Unfortunately, this is not the case. 	

For example, consider {\em locally out-semicomplete digraphs} i.e. digraphs in which the out-neighbors of every vertex induce a semicomplete subdigraph \cite{B}. 
It is well-known that a locally out-semicomplete digraph $D$ has a Hamilton cycle if and only if $D$ is strongly connected. Thus, the Hamilton cycle problem is polynomial-time solvable for locally out-semicomplete digraphs. Now consider a locally out-semicomplete digraph $H$ obtained from a bipartite undirected graph $B$ with partite sets $X$ and $Y$, $|X|=|Y|\ge 2$, by orienting every edge of $B$ from $X$ to $Y$ and adding an arbitrary arc between every pair of vertices in $Y$. Clearly, $H$ has  a Hamilton oriented cycle if and only if $B$ has a Hamilton cycle. However, the problem of deciding whether a bipartite graph has a Hamilton cycle is {\sf NP}-hard and so is the problem of deciding whether $H$ has a Hamilton oriented cycle.

A digraph $D=(V,A)$ is {\em quasi-transitive} if whenever $xy,yz\in A$ for distinct $x,y,z\in V$, we have either $xz\in A$ or $zx\in A$ or both.
We can decide whether a quasi-transitive digraph has a Hamilton cycle in polynomial time \cite{G94}. However, by orienting all edges of a bipartite graph $B$ from $X$ and $Y$ as above, we obtain a quasi-transitive digraph $Q$, which is even a transitive digraph. Thus,  even the problem of deciding whether $Q$ has a Hamilton oriented cycle is {\sf NP}-hard.

\GGn{In fact, the above argument shows that it is {\sf NP}-hard to find a Hamilton oriented cycle in a digraph which is bipartite, acyclic and transitive. }



\end{document}